\documentclass[10pt,fleqn]{amsart}

\usepackage{amsmath, amsfonts, amssymb, amsthm, bm}
\usepackage{hyperref}

\newcommand{\ZZ}{\mathbb{Z}}
\newcommand{\FF}{\mathbb{F}}
\newcommand{\CC}{\mathbb{C}}
\newcommand{\TT}{\mathbb{T}}
\newcommand{\PP}{\mathbb{P}}

\newcommand{\pitilde}{\widetilde{\pi}}

\DeclareMathOperator{\sgn}{sgn}
\DeclareMathOperator{\GL}{GL}
\DeclareMathOperator{\SL}{SL}

\newtheorem{theorem}{Theorem}[section]
\newtheorem{proposition}[theorem]{Proposition}
\newtheorem{lemma}[theorem]{Lemma}
\newtheorem{corollary}[theorem]{Corollary}

\theoremstyle{remark}
\newtheorem{remark}[theorem]{Remark}

\theoremstyle{definition}
\newtheorem{definition}[theorem]{Definition}
\numberwithin{equation}{section}

\title[The Legendre determinant form]{The Legendre determinant form for Drinfeld modules in arbitrary rank}
\author{R. B. Perkins}
\email{perkins@math.univ-lyon1.edu, rudolph.perkins@gmail.com}
\address{Institut Camille Jordan,
Facult\'e de Sciences et Techniques,
23 Rue Dr Paul Michelon,
42023 Saint-Etienne Cedex, France}

\begin{document}

\begin{abstract}
For each positive integer $r$, we construct a nowhere-vanishing, single-cuspidal Drinfeld modular form for $\GL_r(\FF_q[\theta])$, necessarily of least possible weight, via determinants using rigid analytic trivializations of the universal Drinfeld module of rank $r$ and deformations of vectorial Eisenstein series. Along the way, we deduce that the cycle class map from de Rham cohomology to Betti cohomology is an isomorphism for Drinfeld modules of all ranks over $\FF_q[\theta]$.
\end{abstract}

\maketitle

\section{Introduction}
Let $\FF_q$  be the finite field with $q$ elements, and $\theta$ an indeterminate over this field. Let $A, K, K_\infty$, and $A_+$ be $\FF_q[\theta], \FF_q(\theta)$, $\FF_q((\theta^{-1}))$, and the set of monic polynomials in $A$, respectively. Let $\CC_\infty$ be the completion of an algebraic closure of $K_\infty$ equipped with the unique extension $|\cdot|$ of the absolute value making $K_\infty$ complete, and normalized so that $|\theta| := q$. Unless stated otherwise, $r \geq 2$ will denote a fixed positive integer and the juxtaposition of matrices of appropriate sizes indicates their matrix product. For a matrix $A$ of any size, $A^{tr}$ will denote its transpose.

\subsection{A brief history of Drinfeld modular forms of higher rank}
In the published version \cite{GossEis} of his thesis, D. Goss introduced the notion of Drinfeld modular forms for $\SL_r$, building on the rigid-analytic theory for moduli spaces of Drinfeld modules of a fixed rank $r$ originally introduced by Drinfeld. Goss developed several useful tools including Eisenstein series and affinoid coverings of these Drinfeld period domains helpful for describing expansions at the cusps. In rank 2, which is most closely comparable to the classical theory of modular forms and where a compactification of the Drinfeld modular variety was immediately available to him, Goss introduced the first notion of the expansion of a Drinfeld modular form at infinity and worked out many details of the theory in this case.

Sometime later, M. Kapranov \cite{Kap} gave an explicit construction of the Satake compactification of these Drinfeld modular varieties in the $A = \FF_q[\theta]$ setting. With the help of Kapranov, Goss then sketched in \cite{Gossint} how to obtain finite dimensionality of the spaces of modular forms of higher ranks using the Tate theory and Kapranov's constructions.
By different means, Pink \cite{Pink} has constructed Satake compactifications of the period domains of Drinfeld in full generality and, using these compactifications he has introduced a notion of positive characteristic valued algebraic modular forms for $\SL_r$, remarking that these forms should coincide with Goss' analytically defined modular forms from \cite{GossEis}.
To develop the analytic side of this theory, Breuer and Pink introduced a rigid-analytic uniformizer at the cusps for ranks $r \geq 3$, generalizing the uniformizer introduced in the rank 2 setting by Goss.
Basson developed the ideas of Breuer-Pink in his dissertation \cite{Bass} and studied the coefficients of the expansion at infinity for certain higher rank Drinfeld modular forms including the Eisenstein series introduced by Goss and the coefficient forms coming from the universal Drinfeld module of rank $r$.

\subsection{The history of the Legendre determinant form}

In \cite{Gekinv}, Gekeler introduced the notion of Drinfeld modular forms for $\GL_2(A)$ with non-trivial a type character $\pmod{q-1}$, and he introduced a nowhere-vanishing, single-cuspidal modular form $h$ of type 1 and weight $q+1$ in the form of a Poincar\'e series, giving two alternate descriptions of the form. Later, using an idea due to Anderson, a Legendre period relation for Drinfeld modules, Gekeler gave in \cite{Gekcompo} another construction of the form $h$. There he realized $h$ as the reciprocal of the nowhere-vanishing function given by mapping $z \in \CC_\infty \setminus K_\infty$ to the determinant of the matrix $\left( \begin{smallmatrix} z & 1 \\ \eta_1(z) & \eta_2(z) \end{smallmatrix} \right)$, \emph{the Legendre determinant}, where $1,z$ are periods and $\eta_1(z),\eta_2(z)$ are quasi-periods of the lattice $Az + A$.
It turns out that the holomorphy and nowhere-vanishing of this determinant are wrapped up in a cycle class map, also attributed to Anderson, that exists between the de Rham cohomology and the Betti cohomology for a given Drinfeld module of any rank.
Gekeler gave a proof that this map is an isomorphism in \cite{Gekcrelles} via an elegant explicit construction, and he noticed in \cite{Gekcompo} that this Legendre determinant construction for Drinfeld modules of higher ranks $r$ also produces a weakly modular Drinfeld modular form of type 1 and weight $1 +q+\cdots+q^{r-1}$ for $\GL_r(A)$, giving the full details when $r = 2$.

\subsection{Main results}
This note may be seen as a completion of the work done in \cite{Gekcompo}, finishing the proof that Gekeler's Legendre determinant form for $\GL_r(A)$ is holomorphic at infinity for general $r$, while also extending the notion of deformations of vectorial modular forms introduced by Pellarin in \cite{FPannals, FPcrelles} to the rank $r$ setting and demonstrating the role deformations of vectorial Eisenstein series play in the construction of a nowhere-vanishing single cuspidal form of least weight.

We adopt the approach of Pellarin, analyzing the modularity properties of the rigid analytic trivialization of the universal Drinfeld module of rank $r$, to prove the existence of a nowhere-vanishing, single cuspidal Drinfeld modular form of type 1 and weight $1+q+\cdots+q^{r-1}$ for $\GL_r(A)$ in the sense introduced by Breuer-Pink and developed in Basson's dissertation \cite{Bass}; the one-dimensionality of this space is an immediate result of the existence of such a form.
We dub the form so constructed \emph{the Legendre determinant form}.
To accomplish this goal, we introduce the theory of deformations of vectorial Drinfeld modular forms for $\GL_r(A)$ sufficiently to make comparisons with the Basson-Breuer-Pink theory.
We avoid some of the closer analysis of Anderson generating functions done in \cite{Gekcompo} through the use of a $\tau$-difference equation satisfied by the rigid analytic trivialization matrix studied below.
Because the construction of our form is somewhat inexplicit using these trivializations, we also introduce deformations of vectorial Eisenstein series for $\GL_r(A)$ to obtain a form with more precise information about the expansion at infinity, as studied in \cite{Bass}, for comparison with future developments.
As a corollary, we also obtain some non-trivial information about these deformations of vectorial Eisenstein series for $\GL_r(A)$.
Along the way, we give a new proof of the non-singularity of the cycle class isomorphism between the de Rham cohomology and Betti cohomology for Drinfeld modules of arbitrary rank over $A$.

Of course, such a nowhere-vanishing, single cuspidal form, as constructed below, is extremely valuable for inductively determining the algebra structure of the space of Drinfeld modular forms of rank $r$. Therefore, we shall be careful not to use any results on the structure of the aforementioned algebra. We shall only assume the finite dimensionality of the spaces of forms for $\GL_r(A)$ of a fixed weight and type defined below.

It should be said that these ideas are not totally new, already appearing to some extent in \cite{Gekcompo, Gosscoho, FPcrelles}, but with the emergence of more analytic tools for the study of Drinfeld modular forms of higher ranks, it seemed timely to work out the full details in the construction of such a form and to emphasize the connection with the various cohomologies lurking in the background.
Further, it seems to the author that the construction of the form $h$ given below demonstrates an intrinsic reason for its existence, and we hope that this short note will serve as a useful companion to Basson's dissertation and the forthcoming works of Breuer-Pink, Basson-Breuer-Pink and others.

\subsubsection{Acknowledgements} The author would like to thank F. Pellarin for his mentorship and for the opportunity to learn these things. Thanks go to D. Goss for several useful comments on presentation and to D. Basson for helpful conversations.

\section{Deformations of Drinfeld modular forms in arbitrary rank}

\subsection{Drinfeld's rigid analytic period domain}

Let $\Omega^r := \PP^{r-1}(\CC_\infty) \setminus \bigcup H$ be \emph{Drinfeld's period domain}, where the union is over all $K_\infty$-hyperplanes $H$ in $\PP^{r-1}(\CC_\infty)$. The set $\Omega^r$ is a connected (in fact, simply-connected \cite[(3.4) and (3.5)]{VDP}) admissible open subset of $\PP^{r-1}(\CC_\infty)$, and we view it embedded in $M_{1 \times r}(\CC_\infty)$ as row vectors\footnote{Note that we have reversed the indexing for the vector $\bm{z}$ from what appears in Basson's dissertation.} $\bm{z} = (z_r,z_{r-1},\dots,z_1)$, with $z_1 = 1$. The space $\Omega^r$ should be viewed as the moduli space of isogeny classes of Drinfeld modules of rank $r$.

\subsubsection{Action of $\GL_r(A)$}
Following Basson and differing slightly from the standard set in rank two, the group $\Gamma_r := \GL_r(A)$ acts on $\Omega^r \subseteq M_{1 \times r}(\CC_\infty)$ via
\[\gamma \bm{z} := (\bm{z} \gamma^{-1} \bm{e}_r)^{-1} \bm{z} \gamma^{-1}, \text{ for all } \gamma \in \Gamma_r \text{ and } \bm{z} \in \Omega^r,\]
where, from now on, for $i = 1,2,\dots,r$, we write $\bm{e}_i$ for the column vector whose $i$-th entry equals $1$ and whose remaining $r-1$ other entries are all zero. The scalar
\[j(\gamma,\bm{z}) := \bm{z} \gamma^{-1} \bm{e}_r \in \CC_\infty, \]
defined for all $\gamma \in \Gamma_r$ and $\bm{z} \in \Omega$, is a basic factor of automorphy. It does not vanish for any choices of $\bm{z},\gamma$ by the $K_\infty$-linear independence of the elements of $\Omega^r$.

\subsubsection{Imaginary part, an admissible covering}
Given a $K_\infty$-rational hyperplane $H \subseteq \PP^{r-1}(\CC_\infty)$, we choose to represent it by an equation
\[\ell_H(\bm{z}) := \eta_1 z_1+ \eta_2 z_2 + \cdots + \eta_r z_r\]
such that $|\eta_i| \leq 1$, for all $i$ and $|\eta_j| = 1$ for at least one $j$ and where $\bm{z} := (z_r,z_{r-1},\dots,z_1)$. Such a representation is unique up to multiplication by an element in $K_\infty$ of absolute value 1, and the number $|\ell_H(\bm{z})|$ is well-defined.

Let $|\bm{z}| := \max_{i = 1,\dots,r}|z_i|$. Observe that we have required that $z_1 = 1$, and hence $|\bm{z}| \geq 1$, for all $\bm{z} \in \Omega^r$.

Finally, we define the \emph{imaginary part} of an element $\bm{z} \in \Omega^r$ by
\[|\bm{z}|_\Im := \inf_{H}\{ |\ell_H(\bm{z})| \},\]
where the infimum is taken over all $K_\infty$-rational hyperplanes $H \subseteq \PP^{r-1}(\CC_\infty)$, as above. Intuitively, this number measures the ``proximity to infinity'' on $\Omega^r$.

For each non-negative integer $n$, we define
\[\Omega_n = \{\bm{z} \in \Omega^r : |\bm{z}|_\Im \geq q^{-n} |\bm{z}|\}.\]
The sets $\Omega_n$ are affinoid subsets of the rigid analytic space $\Omega^r$, and their collection $\{\Omega_n : n \geq 0 \}$ forms an admissible covering of $\Omega^r$, \cite[Thm. 2.6.12]{Bass}.

\begin{remark}
An admissible covering of a very similar flavor appears in \cite{GossEis}, and a precise connection with the imaginary part defined above and the $d$ function of Goss can be made via Prop. 1.64 of \emph{loc. cit.}.
\end{remark}

\subsection{$\TT$-valued holomorphic functions on $\Omega^r$}

Let $t$ be an indeterminate over $\CC_\infty$, and let $\TT$ be the \emph{Tate algebra}, i.e. the completion of $\CC[t]$ equipped with the \textit{Gauss norm},
\[\displaystyle{\left|\left|\sum f_i t^i \right| \right| := \max \{ |f_i| \}}.\]

Let $\CC_\infty\{\tau\}$ be the non-commutative ring of \emph{twisted polynomials}; see \emph{e.g.} \cite[Ch. 1]{Gbook}. The Tate algebra has the structure of a (left-) $\CC\{\tau\}$-module determined by the continuous action of $\tau$ given by $\tau\left( \sum f_i t^i  \right) := \sum f_i^{q} t^i$.

\begin{definition}
A function $f : \Omega_n \rightarrow \TT$ shall be called \emph{holomorphic on} $\Omega_n$ if it is the uniform limit of rational functions $\TT(z_1,\dots,z_r)$ whose poles lie outside of $\Omega_n$. A function $f: \Omega^r \rightarrow \TT$ shall be called \emph{holomorphic on} $\Omega^r$ if its restriction to $\Omega_n$ is holomorphic for all $n$. We shall say a function $\Omega^r \rightarrow \TT^r$ is \emph{holomorphic on} $\Omega^r$ if each of its coordinates is holomorphic on $\Omega^r$.
\end{definition}

Observe that the holomorphic functions defined above are simply the holomorphic functions from \cite[Def. 2.6.14]{Bass} tensored over $\CC_\infty$ with $\TT$.

\subsection{Parameter at infinity for $\Gamma_r$}
Given a column vector $\bm{a} \in A^{r-1}$, let $[\bm{a}]$ be the square matrix whose first column equals ${0 \choose \bm{a}}$ and whose remaining entries are all zero.
One readily observes that there is an injective group homomorphism $\alpha : A^{r-1} \rightarrow \Gamma_r$ given by
\[ \bm{a} \mapsto [\bm{a}] + [\bm{e}_1 \ \bm{e}_2 \ \cdots \ \bm{e}_r] \text{ for all } \bm{a} \in A^{r-1}. \]
We write $\Gamma_\alpha$ for the image in $\Gamma_r$ of this injection.

Given a lattice $\Lambda \subset \CC_\infty$ of any positive integer rank, there corresponds an exponential function $\mathfrak{e}_\Lambda(w) := w\prod_{\lambda \in \Lambda} \left(1 - w\lambda^{-1}\right)$ which is a $\FF_q$-linear, entire (hence surjective) endomorphism on the additive group of $\CC_\infty$. Given a vector $\bm{z} \in \Omega^r$, we shall write $\bm{z} := (z_r,\widetilde{\bm{z}})$ and consider here the rank $r-1$ lattices $\Lambda(\widetilde{\bm{z}}) := \widetilde{\bm{z}} A^{r-1} \subseteq \CC_\infty$. We make the abbreviation $\mathfrak{e}_{\widetilde{\bm{z}}} := \mathfrak{e}_{\Lambda(\widetilde{\bm{z}})}$.

Since we are working over $A := \FF_q[\theta]$, we follow Goss's convention, extended to ranks $r \geq 3$ by Basson, of normalizing the parameter at infinity to be defined just below with a fundamental period of the Carlitz module $\pitilde := \lambda_\theta \theta \prod_{j \geq 1} ( 1 - \theta^{1-q^j} )^{-1}$
which is unique up to the choice $\lambda_\theta$ of $(q-1)$-th root of $-\theta$ in $\CC_\infty$. One chooses this normalization so that a nice notion of integrality for the coefficient forms of the $u$-expansion appearing below may be developed.

\begin{definition}[Goss $r = 2$, Breuer-Pink $r \geq 3$]
The \emph{parameter at infinity} on $\Omega^r$ for $\Gamma_r$ is the function
\[u_{\widetilde{\bm{z}}}(z_r) := \mathfrak{e}_{\pitilde\widetilde{\bm{z}}}(\pitilde z_r)^{-1} = \pitilde^{-1} \mathfrak{e}_{\widetilde{\bm{z}}}(z_r)^{-1}.\]
\end{definition}
The function $u_{\widetilde{\bm{z}}}(z_r)$ is invariant under the action of $\Gamma_\alpha$ on $\Omega^r$, that is, under the shifts $(z_r,\bm{\widetilde{z}}) \mapsto (z_r - \widetilde{\bm{z}}\widetilde{\bm{a}}, \widetilde{\bm{z}})$, for all column vectors $\widetilde{\bm{a}} \in A^{r-1}$. Their primary use for us is captured by the next result.

\begin{proposition} \label{uexpprop}
Any $\Gamma_\alpha$ invariant rigid holomorphic function $f:\Omega^r \rightarrow \TT$ may be written in the form
\[f(\bm{z}) = \sum_{n \in \ZZ} f_n(\widetilde{\bm{z}}) u_{\widetilde{\bm{z}}}(\bm{z} \bm{e}_1)^n, \]
for a unique choice of holomorphic functions $f_n : \Omega^{r-1} \rightarrow \TT$ whenever $|\bm{z}|_\Im$ is sufficiently large, depending on $f$.
\begin{proof}
This follows readily from \cite[Prop. 3.2.5]{Bass} using the realization of our holomorphic functions as those of Basson tensored over $\CC_\infty$ with $\TT$.
\end{proof}
\end{proposition}

\begin{definition}
We call the expansion guaranteed by the previous proposition for a $\Gamma_\alpha$-invariant function its \emph{$u$-expansion}.
\end{definition}

\subsection{Deformations of Drinfeld modular forms}

Observe that a rigid holomorphic function $f : \Omega^r \rightarrow \TT$ satisfying
\begin{equation}\label{modeq}
f(\gamma \bm{z}) = j(\gamma, \bm{z})^k \det(\gamma)^m f(\bm{z}) \text{ for all } \gamma \in \Gamma_r \text{ and } \bm{z} \in \Omega^r
\end{equation}
is necessarily $\Gamma_\alpha$-invariant, and hence has an expansion as in Prop. \ref{uexpprop}. Such a rigid holomorphic function is uniquely determined by its $u$-expansion since $\Omega^r$ is a connected rigid analytic space.

\begin{definition}
A rigid holomorphic function $f : \Omega^r \rightarrow \TT$ shall be called a \emph{weak deformation of Drinfeld modular forms} of \emph{weight} $k$ and \emph{type} $m \pmod{q-1}$ if $f$ satisfies both \eqref{modeq} and
\begin{equation}\label{holomeq}
f(\bm{z}) = \sum_{i \geq N} f_i(\widetilde{\bm{z}}) u_{\widetilde{\bm{z}}}(\bm{z} \bm{e}_1)^i, \begin{array}{l} \text{ for all } |\bm{z}|_\Im \text{ sufficiently large} \text{ and some } N, \\ \text{ both depending on } f. \end{array}
\end{equation}
We denote the $\TT$-module of such forms by $\bm{M}^!_{k,m}(\TT)$, and the $\CC_\infty$-subspace of $\CC_\infty$-valued forms simply by $\bm{M}^!_{k,m}$.

If, additionally, we have $N \geq 0$ in \eqref{holomeq}, then $f$ shall be called a \emph{deformation of Drinfeld modular forms} of the same weight and type. We denote the $\TT$-module of such forms by $\bm{M}_{k,m}(\TT)$, and the $\CC_\infty$-subspace of $\CC_\infty$-valued forms by $\bm{M}_{k,m}$.

Finally, we say that a deformation of Drinfeld modular forms $f : \Omega^r \rightarrow \TT$ is \emph{single-cuspidal} if the expansion required in \eqref{holomeq} begins at $N = 1$ with $f_1 \neq 0$.
\end{definition}

\begin{remark}
Given a deformation of Drinfeld modular forms $f$, as above, one has that the functions $f_i$ appearing in \eqref{holomeq} will be weak, $\TT$-valued Drinfeld modular forms on $\Omega^{r-1}$ for $\GL_{r-1}(A)$ of the same type as $f$ and respective weights $k-i$ by \cite[Prop. 3.2.7]{Bass}.
\end{remark}

Though the previous definition is necessary for what we do to follow, the next result shows that it does not give more than what Basson has already explored.

\begin{proposition} \label{ddmfthm}
For all weights $k$ and types $m \pmod{q-1}$ we have
\[\bm{M}^!_{k,m}(\TT) \cong \TT \otimes \bm{M}^!_{k,m}.\]
\begin{proof}
Given the finite dimensionality of the spaces of modular forms for $\GL_r(A)$ of a fixed weight and type, the proof of \cite[Lem. 13]{FPannals} goes through, essentially word for word, upon replacing the non-vanishing form $h$ in the proof (whose existence we are trying to demonstrate) with the non-vanishing form $\Delta$.
\end{proof}
\end{proposition}

\section{Existence and nowhere-vanishing} \label{AGFsect}
For this section we fix $\bm{z} \in \Omega^r$. Such a $\bm{z}$ gives rise to the lattice $\bm{z} A^r$ inside $\CC_\infty$ and to the Drinfeld module $\mathfrak{d}^{\bm{z}}$ of rank $r$ determined by
\[\theta \mapsto \mathfrak{d}^{\bm{z}}_\theta := \theta \tau^0 + g_1(\bm{z})\tau + \cdots + g_{r-1}(\bm{z})\tau^{r-1} + \Delta(\bm{z}) \tau^r \in \CC_\infty \{\tau\}.\] We write $\mathfrak{d}^{\bm{z}}_a$ for the image of $a \in A$ under $\mathfrak{d}^{\bm{z}}$.

As before, we denote the exponential function of the lattice $\bm{z} A^r$ by
\[\mathfrak{e}_{\bm{z}}(w) := \sum_{j \geq 0} \alpha_j(\bm{z}) w^{q^j},\]
and we remind the reader that $\mathfrak{d}^{\bm{z}}_a(\mathfrak{e}_{\bm{z}}(w)) = \mathfrak{e}_{\bm{z}}(a w)$, for all $a \in A$.

There are $r$ \emph{Anderson generating functions} associated to the lattice $\bm{z} A^r$. They are
non-zero elements of $\TT$ defined, for $k = 1,2,\dots,r$, by
\[\mathfrak{f}_k(\bm{z};t) := \sum_{j \geq 0} \mathfrak{e}_{\bm{z}}\left( \frac{\bm{z} \bm{e}_k }{\theta^{j+1}} \right) t^j. \]
The most basic property of these functions for us is summarized by the following result, due to Pellarin \cite{FPbourbaki}; see also \cite[Prop. 6.2 and Rem. 6.3]{EGPmonmath}. To follow, we occasionally suppress the notation $\bm{z}$ and $t$, considering them fixed.

\begin{lemma} \label{FPlem}
The functions $\mathfrak{f}_1,\mathfrak{f}_2, \dots, \mathfrak{f}_r$ are rigid-holomorphic functions $\Omega^r \rightarrow \TT$ which may be expanded, for $k = 1,\dots,r$, as
\[\mathfrak{f}_k(\bm{z};t) = \sum_{j \geq 0} \frac{\alpha_j(\bm{z})(\bm{z} \bm{e}_k)^{q^j}}{\theta^{q^j} - t}.\]
Further, the set $\{\mathfrak{f}_1,\mathfrak{f}_2, \dots, \mathfrak{f}_r\}$ gives a basis for the rank $r$ $\FF_q[t]$-submodule of $\TT$ of solutions to the $\tau$-difference equation
\[\mathfrak{d}_\theta(X) = t X.\]
\end{lemma}

\begin{remark} \label{ATrmk}
Let $\lambda_\theta$ be a $(q-1)$-th root of $-\theta$. When $r = 1$, \emph{the Anderson-Thakur function},
\[\omega(t) := \lambda_\theta \prod_{i \geq 0} (1 - t \theta^{-q^i})^{-1},\]
which debuted in in \cite{AT90}, generates the free $\FF_q[t]$-submodule of $\TT$ of solutions to
\[\tau(X) = (t - \theta)X.\]

In particular, any non-zero solution in $\TT$ of the previous $\tau$-difference equation necessarily has a simple pole at $t = \theta$.
\end{remark}

As a corollary of Lemma \ref{FPlem}, for all $a \in A$ and $k = 1,2,\dots,r$, we have,
\begin{equation} \label{eigenprop} \mathfrak{d}_a(\mathfrak{f}_k) = \chi_t(a) \mathfrak{f}_k,\end{equation}
where $\chi_t(a)$ denotes the image of $a \in A$ under the map $\chi_t: A \rightarrow \FF_q[t] \subseteq \TT$ of $\FF_q$-algebras determined by $\theta \mapsto t$.

We extend the action of $\tau$ to matrices with entries in $\TT$ by $\tau[a_{ij}] := [\tau(a_{ij})]$. Let
\begin{equation*}
\Phi := \left[ \begin{array}{cccc} 0 & 1 & \cdots & 0 \\ \vdots & \vdots & \ddots & \vdots \\ 0 & 0 & \cdots & 1 \\ (t - \theta)/\Delta & -g_1/\Delta & \cdots & -g_{r-1}/\Delta \end{array} \right] \text{ and } \Psi := \left[ \begin{array}{c} \mathcal{F} \\ \tau\mathcal{F} \\ \vdots \\ \tau^{r-1}\mathcal{F} \end{array} \right],
\end{equation*}
with $\mathcal{F} :=  (\mathfrak{f}_1, \mathfrak{f}_2, \dots, \mathfrak{f}_r)$. The previous lemma gives the following matrix identity which will be useful to follow,
\begin{equation} \label{RATdiffeq}
\tau\Psi = \Phi \Psi;
\end{equation}
see \cite[\S 4.2.2]{FPbourbaki} for the connection of \eqref{RATdiffeq} with the dual Anderson $t$-motive associated to $\mathfrak{d}^{\bm{z}}$.

\subsection{Modularity of $\mathcal{F}$}
Let $\rho_t : M_{k \times l}(A) \rightarrow M_{k \times l}(\FF_q[t])$ be the map defined by
\[[a_{ij}] \mapsto [\chi_t(a_{ij})], \text{ for all } [a_{ij}] \in M_{k \times l}(A).\]

The next result is the obvious extension to $\GL_r(A)$ of Pellarin's \cite[Lem. 2.4]{FPcrelles}. It also generalizes Lemma 4.4 and equation (4.3) of \cite{Gekcompo}.

\begin{lemma} \label{AGFcoordeqnlem}
The following identity holds for all $\gamma \in \Gamma_r$ and $\bm{z} \in \Omega^r$,
\[\mathcal{F}(\gamma \bm{z}) = j(\gamma,\bm{z})^{-1} \mathcal{F}(\bm{z}) \rho_t(\gamma)^{-1}.\]
\begin{proof}
On each coordinate we have
\begin{eqnarray*}
\mathfrak{f}_k(\gamma \bm{z};t) &=& \sum_{j \geq 0} \frac{\alpha_j(\gamma \bm{z}) (\gamma\bm{z} \bm{e}_k)^{q^j}}{\theta^{q^j} - t} \\
&=& \sum_{j \geq 0} \frac{ j(\gamma,\bm{z})^{q^j - 1}\alpha_j(\bm{z}) ( j(\gamma,\bm{z})^{-1} \bm{z} \gamma^{-1} \bm{e}_k )^{q^j}}{\theta^{q^j} - t} \\
&=& j(\gamma,\bm{z})^{- 1} \mathcal{F}(z) \rho_t(\gamma)^{-1} \bm{e}_k.
\end{eqnarray*}
From the first to the second line we use the modularity properties of the $\alpha_j$ from \cite[Cor. 3.4.9]{Bass}, and from the second to the third we use $\FF_q$-linearity and \eqref{eigenprop} for the Anderson generating functions.
\end{proof}
\end{lemma}

\subsection{Moore determinants}\label{mooredetsect}
Pellarin already noted in \cite{FPcrelles} that $\det \Psi(\bm{z})$ was non-zero and invertible in the fraction field of $\TT$ for all $\bm{z} \in \Omega^r$. We make a quick digression into the theory of Moore determinants to sketch a proof of this non-vanishing. Let $\mathbb{L}$ be the fraction field of the integral domain $\TT$. The $\FF_q[t]$-algebra action of $\tau$ on $\TT$ extends to an $\FF_q(t)$-algebra action on $\mathbb{L}$, and the fixed field of this action is exactly $\FF_q(t)$. The arguments of \cite[\S 1.3]{Gbook} readily extend to our setting, and we have the following result.

\begin{lemma} \label{mooredetlem}
For a Moore matrix
\[\mathcal{M} := \left[ \begin{array}{rrrr} m_1 & m_2 & \cdots & m_r \\ \tau(m_1) & \tau(m_2) & \cdots & \tau(m_r) \\ \vdots & \vdots & & \vdots \\ \tau^{r-1}(m_1) & \tau^{r-1}(m_2) & \cdots & \tau^{r-1}(m_r) \end{array} \right],\]
with $m_1, m_2, \dots, m_r \in \mathbb{L}$, we have $\det\mathcal{M} \neq 0$ if and only if $m_1,m_2,\dots,m_r$ are $\FF_q(t)$-linearly independent.
\end{lemma}

\subsection{Non-vanishing of $\det \Psi$ and cohomology}
This paragraph should be compared with \cite{Gosscoho} of Goss, where these ideas are attributed to Anderson.

Let $\mathfrak{d}$ be a fixed Drinfeld module of rank $r$ with associated exponential function $e_{\mathfrak{d}}$. In the $A = \FF_q[\theta]$ setting in which we are working, one may show that the $r$ $\FF_q$-linear biderivations $\delta^{(0)}, \delta^{(1)},\dots,\delta^{(r-1)}$ determined by mapping $\theta$ to $\mathfrak{d}_\theta - \theta, \tau, \dots, \tau^{r-1} \in \CC_\infty\{\tau\}$, respectively, form a basis for the de Rham cohomology $H^1_{\text{DR}}(\mathfrak{d})$ of $\mathfrak{d}$ as a $\CC_\infty$-vector space; see \cite{Gekcrelles} for definitions and this remark. Uniquely associated to these $r$ biderivations are the $\FF_q$-linear entire functions $F_{0},F_{1},\dots,F_{r-1}$ given by
\[e_{\mathfrak{d}}(w) - w, \sum_{i \geq 0} \theta^i e_{\mathfrak{d}}(w \theta^{-i-1})^q,\dots,\sum_{i \geq 0} \theta^i e_{\mathfrak{d}}(w \theta^{-i-1})^{q^{r-1}},\]
respectively. These functions are additionally $A$-linear upon restriction to the kernel $\Lambda$ in $\CC_\infty$ of $e_{\mathfrak{d}}$, which we think of as the first homology group of $\mathfrak{d}$. The association $\delta^{i} \mapsto F_i$ gives a map, the \emph{de Rham morphism} or \emph{cycle class map}, $\text{DR}: H^1_{\text{DR}}(\mathfrak{d}) \rightarrow \text{Hom}_A (\Lambda, \CC_\infty)$ of $\CC_\infty$-vector spaces.

\begin{corollary}
For all Drinfeld modules of rank $r \geq 2$, the de Rham morphism $\emph{DR}$ is an isomorphism.
\begin{proof}
Let $\mathfrak{d}$ be a Drinfeld module of rank $r$, and let $\bm{z} \in \Omega^r$ represent the isogeny class of $\mathfrak{d}$. As we have indicated above the functions $\mathfrak{f}_1(\bm{z}), \mathfrak{f}_2(\bm{z}), \dots, \mathfrak{f}_r(\bm{z})$ are $\FF_q[t]$-linearly independent elements of $\TT$.
Thus by Lemma \ref{mooredetlem}, we have that $\det\Psi(\bm{z})$ is a non-zero element of $\TT$.
\emph{Ad hoc}, we let $\delta$ be a fixed $(q-1)$-th root in $\CC_\infty$ of $\Delta$. Then, by \eqref{RATdiffeq} we have that $\tau (\delta \det \Psi(\bm{z})) = (t - \theta) \delta \det \Psi(\bm{z})$.
By Remark \ref{ATrmk}, we deduce that $\det\Psi(\bm{z})$ has a simple pole at $t = \theta$; in particular, the residue of this pole is non-zero. One easily checks that the residue at $t = \theta$ of $\det\Psi(\bm{z})$ is equal to the determinant of $[F_i(z_j)]_{i,j}$. Thus we deduce that $\det[F_i(z_j)]_{i,j}$ is a non-zero element of $\CC_\infty$, which implies the claim.
\end{proof}
\end{corollary}

\subsection{Existence and nowhere-vanishing}
For $a_1,a_2, \dots, a_r \in \TT$, define the $r \times r$ diagonal matrices $\text{diag}(a_1,a_2,\dots,a_r) := [b_{ij}]$, with $b_{ii} = a_i$ and $b_{ij} = 0$, for $i \neq j$.

\begin{theorem} \label{nowherevan}
There exists a nowhere-vanishing, single cuspidal Drinfeld modular form for $\GL_r(A)$ of weight $1+q+\cdots+q^{r-1}$ and type $1$.

In particular, $\bm{M}_{\frac{q^r-1}{q-1},1}$ is a $\CC_\infty$-vector space of dimension one.
\end{theorem}
\begin{proof}
Again, the $\FF_q[t]$-linear independence of the functions $\mathfrak{f}_1(\bm{z})$, $\mathfrak{f}_2(\bm{z})$, $\dots, \mathfrak{f}_r(\bm{z})$, for each $\bm{z} \in \Omega^r$ in combination with Lemma \ref{mooredetlem} implies that $\det\Psi(\bm{z})$ is a non-zero element of $\TT$, i.e. $\det\Psi$ is a nowhere-vanishing rigid holomorphic function from $\Omega^r$ to $\TT$. Thus, $\det\Psi^{-1}$ is a well-defined, nowhere-vanishing rigid-holomorphic function from $\Omega^r$ to $\mathbb{L}$. Further, from Lemma \ref{AGFcoordeqnlem}, we obtain
\[\Psi^{-1}(\gamma \bm{z}) = \rho_t(\gamma) \Psi^{-1}(\bm{z}) \text{diag}(j(\gamma,\bm{z}),j(\gamma,\bm{z})^{q},\dots,j(\gamma,\bm{z})^{q^{r-1}}).\]
After taking determinants, the previous identity and Prop. \ref{ddmfthm}, with coefficients extended to $\mathbb{L}$, prove the existence of a nowhere-vanishing element of $\mathbb{L}\otimes\bm{M}^!_{\frac{q^r-1}{q-1},1}$.

It remains to understand the behavior at infinity of the function $\det\Psi^{-1}$. From \eqref{RATdiffeq}, we obtain the equality of norms $||\det\Psi^{-1}(\bm{z})||^{q-1} = ||\Delta(\bm{z})/(t-\theta)||$, for all $\bm{z} \in \Omega^r$. The order of vanishing in $u$ of $\Delta$ is $q-1$, by Basson's product formula for $\Delta$, \cite[Theorem 3.5.14]{Bass}. Thus we see that $||\det\Psi^{-1}||$ grows like $|u_{\widetilde{\bm{z}}}(z_r)|$, for $\widetilde{\bm{z}}$ fixed, as $|\bm{z}|_{\Im} \rightarrow \infty$, proving that $\det\Psi^{-1}$ belongs to $\mathbb{L} \otimes \bm{M}_{\frac{q^r-1}{q-1},1}$ and is single-cuspidal. This finishes the proof of existence of a single-cuspidal, nowhere-vanishing element of $\bm{M}_{\frac{q^r-1}{q-1},1}$; call it $\eta$.

Any element of $\bm{M}_{\frac{q^r-1}{q-1},1}$ is necessarily cuspidal, and the obvious 1-dimensionality of $\bm{M}_{0,0}$ combined with the nowhere-vanishing and single-cuspidality of $\eta$ imply that the former space is 1-dimensional, as desired.
\end{proof}

\section{Deformations of vectorial Eisenstein series}
With the one-dimensionality of $\bm{M}_{1+q+\cdots q^{r-1},1}$ guaranteed by Theorem \ref{nowherevan} above, we now rephrase everything in terms of deformations of vectorial Eisenstein series for $\GL_r(A)$ to give an alternate construction of this form whose coefficient of $u$ may be explicitly determined.

\subsection{The set-up}
We recall that we have let\footnote{Here we emphasize the non-canonical choice of algebra generator for the ring $A$ and the dependence on this choice of what we do to follow.} $\chi_t: A \rightarrow \FF_q[t] \subseteq \TT$ be the morphism of $\FF_q$-algebras determined by $\theta \mapsto t$ and $\rho_t : M_{k \times l}(A) \rightarrow M_{k \times l}(\FF_q[t])$ the map defined by
\[[a_{ij}] \mapsto [\chi_t(a_{ij})], \text{ for all } [a_{ij}] \in M_{k \times l}(A).\]
Observe that if $\gamma \in M_{k \times l}(A)$ and $\gamma' \in M_{l \times m}(A)$, then $\rho_t(\gamma \gamma') = \rho_t(\gamma)\rho_t(\gamma')$.
Finally, set $A^r := M_{r \times 1}(A)$ and $\TT^r := M_{r \times 1}(\TT)$.

\begin{definition}
Let $k$ be a positive integer. The \emph{vectorial Eisenstein series} of weight $k$ associated to the representation $\rho_t: \GL_r(A) \rightarrow \GL_r(\FF_q[t])$ is the function $\mathcal{E}_k : \Omega^r \rightarrow \TT^r$ given by
\[\mathcal{E}_k(\bm{z}) := \sideset{}{'}\sum_{\bm{a} \in A^r} (\bm{z}\bm{a})^{-k} \rho_t(\bm{a});\]
here the primed summation signifies the exclusion of the zero column vector.
\end{definition}

We will see just below that $\mathcal{E}_{q^i}$ is non-zero whenever for all non-negative integers $i$. With a little more work and the introduction of Goss polynomials, it can be shown that $\mathcal{E}_k$ is non-zero whenever $k \equiv 1 \pmod{q-1}$. Since we will not use this here, we do not bother with the details. Now we deal will holomorphy and weak modularity of these functions.

\subsubsection{Holomorphy}
For $\bm{a} \in A^r$, write $\bm{a} = (a_r,\dots,a_1)^{tr}$ and let $|\bm{a}| := \max_{i = 1,\dots,r} |a_i|$. Assume, for some positive integer $n$, that $|\bm{z}|_\Im > q^{-n} |\bm{z}|$, i.e. that $\bm{z}$ is in the admissible open subset $\Omega_n$. We examine the convergence of the coordinates of $\mathcal{E}_k$, which, for all $i = 1,\dots, r$, are of the form $\sum_{\bm{a} \in A^r}' (\bm{z}\bm{a})^{-k} \chi_t(\bm{a}^{tr}\bm{e}_i)$. By our assumption on the size of $\bm{z}$, the general term has size
\[||(\bm{z}\bm{a})^{-k} \chi_t(a_i)|| < |\bm{a}|^{-k} q^{kn} |\bm{z}|^{-k} \leq q^{nk} |\bm{a}|^{-k};\]
the last inequality follows from the assumption that $z_1 = 1$. Since there are only finitely many elements in $A^r$ of a given absolute value, we see that the general term tends to zero, giving the desired uniform convergence on $\Omega_n$.

\subsubsection{Weak Modularity}
\emph{For all $\bm{z} \in \Omega$ and $\gamma \in \Gamma_r$,}
\begin{eqnarray*}
\mathcal{E}_k(\gamma \bm{z}) &=& \sideset{}{'}\sum_{\bm{a} \in A^r}  (j(\gamma,\bm{z})^{-1}\bm{z}\gamma^{-1}\bm{a})^{-k} \rho_t(\bm{a}) \\
&=& j(\gamma,\bm{z})^{k}\rho_t(\gamma)\sideset{}{'}\sum_{\bm{a} \in A^r} (\bm{z}\gamma^{-1}\bm{a})^{-k} \rho_t(\gamma^{-1}\bm{a}) \\
&=& j(\gamma,\bm{z})^{k}\rho_t(\gamma) \mathcal{E}_j(\bm{z}).
\end{eqnarray*}

\subsubsection{The coordinates of $\tau^k(\mathcal{E}_1)$}
Here we look closely at the coordinate functions of the subfamily $\mathcal{E}_{1},\mathcal{E}_{q},\mathcal{E}_{q^2},\dots$ to show that they are non-zero. We extend the action of $\tau$ to $\TT^r$ coordinate-wise, as was done with matrices above. We extend the action of $\tau$ to functions $f : \Omega^r \rightarrow \TT$ by defining $\tau(f) : \Omega^r \rightarrow \TT$ to be the map $\bm{z} \mapsto \tau(f(\bm{z}))$, for each $\bm{z} \in \Omega^r$, and similarly for functions $\Omega^r \rightarrow \TT^r$.

Define the coordinates of the Eisenstein series
\[\mathcal{E}_{1} := (\epsilon_r,\epsilon_{r-1},\dots,\epsilon_1)^{tr},\]
and, for $\bm{a} \in A^{r}$, write $\bm{a} := {a_r \choose \widetilde{\bm{a}}}$, with $\widetilde{\bm{a}} := (a_{r-1},a_{r-2},\dots,a_{1})^{tr} \in A^{r-1}$. Observe that, splitting the sum defining $\mathcal{E}_1$ according to whether $a_r$ is zero or not, for each non-negative integer $k$ we may write
\[\tau^k(\epsilon_r)(\bm{z}) = -\sum_{a_r \in A_+} \sum_{\widetilde{\bm{a}} \in A^{r-1}} \frac{\chi_t(a_r)}{(a_r z_r + \widetilde{\bm{z}}\widetilde{\bm{a}})^{q^k}} = -\pitilde^{q^k} \sum_{a_r \in A_+} \chi_t(a_r) u_{\widetilde{\bm{z}}}(a_r z_r)^{q^k},\]
\[\tau^k(\epsilon_i)(\bm{z}) = \sideset{}{'}\sum_{\widetilde{\bm{a}} \in A^{r-1}} \frac{\chi_t(a_i)}{(\widetilde{\bm{z}}\widetilde{\bm{a}})^{q^k}} -\sum_{a_r \in A_+} \sum_{\widetilde{\bm{a}} \in A^{r-1}} \frac{\chi_t(a_i)}{(a_r z_r + \widetilde{\bm{z}}\widetilde{\bm{a}})^{q^k}},\]
for $i = 1,2,\dots,r-1$.

We readily obtain the following result.

\begin{proposition} \label{limprop}
For all $k \geq 0$, the function $\tau^k(\mathcal{E}_1)$ is not identically zero.

Further, for fixed $\widetilde{\bm{z}} \in \Omega^{r-1} \subseteq \CC_\infty^{r-1}$, as $|\bm{z}|_\Im \rightarrow \infty$, the order of vanishing in $u_{\widetilde{\bm{z}}}$ of $\tau^k(\epsilon_r)$ equals $q^k$, and
\[\tau^k(\epsilon_j) \rightarrow \tau^k(\widetilde{\epsilon_j}) := \sideset{}{'}\sum_{\widetilde{\bm{a}} \in A^{r-1}} (\widetilde{\bm{z}}\widetilde{\bm{a}})^{-q^k}\chi_t(a_j) \text{ for all } j = 1,2,\dots,r-1.\]
\end{proposition}

\subsection{The Main Result}
Define $\Xi := \left[ \mathcal{E}_1 \ \mathcal{E}_q \ \cdots \ \mathcal{E}_{q^{r-1}} \right]$, and observe that $\tau^k(\mathcal{E}_1) = \mathcal{E}_{q^k}$, for all $k \geq 0$. Set
\[\widetilde{\mathcal{E}}_1 := (\widetilde{\epsilon}_{r-1}, \dots, \widetilde{\epsilon}_1)^{tr} \text{ and } \widetilde{\Xi} := [ \widetilde{\mathcal{E}}_1 \ \tau(\widetilde{\mathcal{E}}_1) \ \cdots \ \tau^{r-2}(\widetilde{\mathcal{E}}_1) ]. \]

\begin{theorem} \label{refinedmainthm}
The function $\det\Xi$ is a nowhere-vanishing, single-cuspidal deformation of Drinfeld modular forms in $\TT^\times \otimes \bm{M}_{1+q+\cdots+q^{r-1},1}$.

When $r = 2$, the coefficient of $u$ in $\det\Xi$ equals $\pitilde L(\chi_t,q)$, and when $r \geq 3$, it equals  $-\pitilde \tau(\det \widetilde{\Xi})$.
\end{theorem}
\begin{proof}
The proof of Theorem \ref{refinedmainthm} shall proceed by induction. We establish the base-case now.

\subsubsection{Construction in rank two}
Everything, except nowhere-vanishing, for the realization of the title is already contained in \cite{FPannals}. Note that there are some superficial differences between Pellarin's setting and ours due to the different action of $\GL_2(A)$ on $\Omega^2$.

From \cite[Lemma 14]{FPannals} we have that $\det \Xi$ is weakly modular of weight $q+1$ and type $1$. Upon proving that $\det \Xi$ is well-behaved at infinity we will conclude that it is a $\TT$-valued modular form. Explicitly,
\begin{eqnarray}
\label{deteq1} \det \Xi &=& \left( \sideset{}{'}\sum_{a,b \in A} (az+b)^{-1} \chi_t(a) \right)\left( \sideset{}{'}\sum_{a,b \in A} (az+b)^{-q} \chi_t(b) \right) \\
\label{deteq2} && - \left( \sideset{}{'}\sum_{a,b \in A} (az+b)^{-1} \chi_t(b) \right)\left( \sideset{}{'}\sum_{a,b \in A} (az+b)^{-q} \chi_t(a) \right).\end{eqnarray}
From \cite[Lemma 21]{FPannals}, we see that the first sum on the right side of \eqref{deteq1} is divisible in $\TT[[u]]$ by $u(z) := \sum_{a \in A} (\pitilde z + \pitilde a)^{-1}$, but not by $u^2$, and similarly, the second sum in \eqref{deteq2} is divisible in $\TT[[u]]$ by $u^q$. From \cite[Lemma 25]{FPannals}, we see that when $|z| = |z|_\Im := \inf_{\kappa \in K_\infty} |z - \kappa|$ tends to infinity the second sum in \eqref{deteq1} tends to $-L(\chi_t,q) := -\sum_{a \in A_+} a^{-q} \chi_t(a)$ while the first sum in \eqref{deteq2} tends to $-L(\chi_t,1) := -\sum_{a \in A_+} a^{-1} \chi_t(a)$. Thus, from the $u$-expansion we see that $\det \Xi$ is not-identically zero, is well behaved at infinity, and that the coefficient of $u$ equals $\pitilde L(\chi_t,q) \in \TT^\times$. After Thm. \ref{nowherevan} above we conclude that $\pitilde^{-1}L(\chi_t,q)^{-1}\det\Xi$ is a nowhere-vanishing, single-cuspidal Drinfeld modular form of weight $1+q$ and type $1$ in rank $2$, whose coefficient of $u$ equals $1$. \hfill $\qed$

\subsubsection{The inductive step}
The transformation properties of the columns of $\Xi$ are summarized by the following matrix identity,
\[\Xi(\gamma \bm{z}) = \rho_t(\gamma) \Xi(\bm{z})\text{diag}(j(\gamma,\bm{z}),j(\gamma,\bm{z})^{q},\dots,j(\gamma,\bm{z})^{q^{r-1}}),\]
which we have seen holds for all $\bm{z} \in \Omega^r$ and $\gamma \in \Gamma_r$.
Taking determinants we observe that $\det\Xi$ is a weak deformation of Drinfeld modular forms of weight $1+q+\cdots+q^{r-1}$ and type $1$. For each $j = 1,2,\dots,r-1$, the entries of $\tau^j(\mathcal{E}_1)$ have well-defined limits as $|\bm{z}|_\Im \rightarrow \infty$ with $\widetilde{\bm{z}}$ fixed, and we observe  $\det \Xi$ has a $u_{\widetilde{\bm{z}}}$-expansion with no negative terms by \cite[Prop. 3.2.9]{Bass} and Prop. \ref{ddmfthm} above. Hence, $\det \Xi$ is an element of $\TT \otimes \bm{M}_{1+q+\cdots+q^{r-1},1}$.

To see that $\det\Xi$ is non-zero, we look closely at its $u$-expansion. We employ the well-known Leibniz formula for the determinant,
\[ \det \Xi = \sum_{\sigma \in S_r} \sgn(\sigma) \prod_{i = 1}^r \tau^{\sigma(i)-1} (\epsilon_i);\]
here the sum is over all elements of the symmetric group on $r$ letters, and $\sgn$ is the usual sign function.
By the Prop. \ref{limprop}, we may focus our attention on those $\sigma \in S_r$ such that $\sigma(r) = 1$ in order to determine the coefficient of $u_{\widetilde{\bm{z}}}$. A moment's thought gives that the coefficient of $u_{\widetilde{\bm{z}}}$ in $\det \Xi$ is given by
\[-\pitilde\sum_{\substack{\sigma \in S_r \\ \sigma(r) = 1}} \sgn(\sigma) \prod_{i = 1}^{r-1} \tau^{\sigma(i)-1} (\widetilde{\epsilon}_i(\widetilde{\bm{z}})) = -\pitilde \tau(\det \widetilde{\Xi}(\widetilde{\bm{z}})).\]
It follows inductively by the calculation done in rank $2$ that $\tau(\det\Xi)$ is an element of $\TT^\times\otimes \bm{M}_{1+q+\cdots+q^{r-1},1}$ and, in particular, that it is non-zero. Thm. \ref{nowherevan} implies that it is nowhere-zero and finishes the proof.
\end{proof}

\begin{remark}
In the rank 2 setting, we may now compare $\det\Xi$ with Gekeler's form $h$, normalized as in \cite{Gekinv}. We have $\det \Xi = -\pitilde L(\chi_t,q) h$.
\end{remark}

\begin{remark}[Construction of all single-cuspidal forms in rank 2] \label{snglcsprmk}
The same idea as in the construction of $h$ above may be used to show that in rank 2 the form $\det [ \mathcal{E}_1 \ \mathcal{E}_{1+k(q-1)} ]$ is a non-zero, single-cuspidal form of weight $2 + k(q-1)$, for all $k \geq 1$. This captures all of the single-cuspidal forms in for $\GL_2(A)$. We shall pose some questions about this in arbitrary rank below.
\end{remark}

\subsection{Corollaries for deformations of vectorial Eisenstein series}
As an immediate consequence of Theorem \ref{refinedmainthm} above
we obtain the following result.

\begin{corollary}
Let $\mathcal{E}_1 := (\epsilon_r,\epsilon_{r-1},\dots,\epsilon_1)^{tr}$. Then for all $i = 1,2,\dots,r$ and all $\bm{z} \in \Omega^r$ we have $\epsilon_i(\bm{z}) \neq 0$. In particular, the vectorial modular form $\mathcal{E}_1$ is nowhere-vanishing on $\Omega^r$.
\end{corollary}

\begin{remark}
Conversely, the non-vanishing of $\epsilon_i(\bm{z})$ for all $i = 1,2,\dots,r$ and all $\bm{z} \in \Omega^r$ combined with the $\FF_q[t]$-linear independence of $\epsilon_1(\bm{z}),\epsilon_2(\bm{z}),\dots,\epsilon_r(\bm{z})$ for all $\bm{z} \in \Omega^r$ and an application of the theory of Moore determinants in our setting (see \S \ref{mooredetsect} below) is enough to guarantee the nowhere-vanishing of $\det\Xi$. Since the verification of these sufficient conditions appears difficult using the definition of $\mathcal{E}_1$ given above, this explains the needs for the Anderson generating functions and Moore determinants employed below.
\end{remark}

As another consequence of Theorem \ref{refinedmainthm} above, we obtain non-trivial dependence relations between vectorial Eisenstein series of low weights. We thank F. Pellarin for this remark.

\begin{corollary} \label{EScor1}
Suppose $k_1,k_2,\dots,k_r$ are positive integers all congruent to $1$ modulo $q-1$ such that $k_1+k_2+\cdots+k_r < 1+q+\cdots+q^{r-1}$. Then $\det[ \mathcal{E}_{k_1} \ \mathcal{E}_{k_2} \ \cdots \ \mathcal{E}_{k_r}] = 0$.
\end{corollary}
\begin{proof}[Sketch]
From the work done above, the form $\det[ \mathcal{E}_{k_1} \ \mathcal{E}_{k_2} \ \cdots \ \mathcal{E}_{k_r}]$ will be a cuspidal Drinfeld modular form of weight $k_1+k_2+\cdots+k_r$. Thus dividing by the non-vanishing form $\det \Xi$, we obtain a form of negative weight which is holomorphic at infinity, in Basson's sense. The space of such forms consists of zero alone.
\end{proof}

\begin{remark}
Conversely, in rank 2 we see that the deformations of vectorial modular forms $g\mathcal{E}_1$ and $\mathcal{E}_q$ both have weight $q$, and they are $\TT$-linearly independent by the non-vanishing of $\det \Xi$.
\end{remark}

\section{Conclusion}
We are left with several questions in closing, and we hope to return to them at a future time.

1. The exact relation ship between the columns of the matrix $\Psi^{-1}$ and of $\Xi$ must be determined, as was done by Pellarin in the rank 2 setting in \cite{FPannals}.

2. The flexibility of the variable $t$, in particular that it may be evaluated in $\CC_\infty$, suggests the existence of unexplored de Rham cohomologies, extending the notion developed by Gekeler \emph{et. al.}. Indeed, the de Rham cohomology which exists presently solely uses the usual left-action of multiplication by elements of $A$ on the twisted polynomial ring $\CC_\infty\{\tau\}$, where the biderivations described above take their values. Perhaps carefully chosen evaluations of the variable $t$ in $\CC_\infty$ will yield interesting actions of $A$ on $\CC_\infty\{\tau\}$ producing new non-trivial information for the Drinfeld module at hand.

3. The Anderson generating functions $\mathfrak{f}_1, \mathfrak{f}_2, \dots, \mathfrak{f}_r$ may be viewed as elements of the $\theta$-adic Tate module associated to the Drinfeld module $\mathfrak{d}$. Since these Anderson generating functions also arise in connection with the de Rham cohomology of $\mathfrak{d}$, it would be interesting to make a precise relationship between de Rham cohomology and the $\theta$-adic Tate module.

4. It would be interesting to know if all single-cuspidal Drinfeld modular forms in arbitrary rank $r$ may be obtained as the determinant of some subset of $r$ distinct deformations of vectorial Eisenstein series, as was the case in rank 2 mentioned in Remark \ref{snglcsprmk}.

\end{document}